\documentclass[12pt]{amsart}

\usepackage{amssymb}

\usepackage[mathscr]{eucal}

\input xypic
\xyoption{all}

\makeindex
\makeglossary

\begin{document}
\baselineskip = 16pt

\newcommand \ZZ {{\mathbb Z}}
\newcommand \NN {{\mathbb N}}
\newcommand \RR {{\mathbb R}}
\newcommand \CC {{\mathbb C}}
\newcommand \PR {{\mathbb P}}
\newcommand \AF {{\mathbb A}}
\newcommand \GG {{\mathbb G}}
\newcommand \QQ {{\mathbb Q}}
\newcommand \bcA {{\mathscr A}}
\newcommand \bcC {{\mathscr C}}
\newcommand \bcD {{\mathscr D}}
\newcommand \bcF {{\mathscr F}}
\newcommand \bcG {{\mathscr G}}
\newcommand \bcH {{\mathscr H}}
\newcommand \bcM {{\mathscr M}}
\newcommand \bcJ {{\mathscr J}}
\newcommand \bcL {{\mathscr L}}
\newcommand \bcO {{\mathscr O}}
\newcommand \bcP {{\mathscr P}}
\newcommand \bcQ {{\mathscr Q}}
\newcommand \bcR {{\mathscr R}}
\newcommand \bcS {{\mathscr S}}
\newcommand \bcV {{\mathscr V}}
\newcommand \bcW {{\mathscr W}}
\newcommand \bcX {{\mathscr X}}
\newcommand \bcY {{\mathscr Y}}
\newcommand \bcZ {{\mathscr Z}}
\newcommand \goa {{\mathfrak a}}
\newcommand \gob {{\mathfrak b}}
\newcommand \goc {{\mathfrak c}}
\newcommand \gom {{\mathfrak m}}
\newcommand \gon {{\mathfrak n}}
\newcommand \gop {{\mathfrak p}}
\newcommand \goq {{\mathfrak q}}
\newcommand \goQ {{\mathfrak Q}}
\newcommand \goP {{\mathfrak P}}
\newcommand \goM {{\mathfrak M}}
\newcommand \goN {{\mathfrak N}}
\newcommand \uno {{\mathbbm 1}}
\newcommand \Le {{\mathbbm L}}
\newcommand \Spec {{\rm {Spec}}}
\newcommand \Gr {{\rm {Gr}}}
\newcommand \Pic {{\rm {Pic}}}
\newcommand \Jac {{{J}}}
\newcommand \Alb {{\rm {Alb}}}
\newcommand \Corr {{Corr}}
\newcommand \Chow {{\mathscr C}}
\newcommand \Sym {{\rm {Sym}}}
\newcommand \Prym {{\rm {Prym}}}
\newcommand \cha {{\rm {char}}}
\newcommand \eff {{\rm {eff}}}
\newcommand \tr {{\rm {tr}}}
\newcommand \Tr {{\rm {Tr}}}
\newcommand \pr {{\rm {pr}}}
\newcommand \ev {{\it {ev}}}
\newcommand \cl {{\rm {cl}}}
\newcommand \interior {{\rm {Int}}}
\newcommand \sep {{\rm {sep}}}
\newcommand \td {{\rm {tdeg}}}
\newcommand \alg {{\rm {alg}}}
\newcommand \im {{\rm im}}
\newcommand \gr {{\rm {gr}}}
\newcommand \op {{\rm op}}
\newcommand \Hom {{\rm Hom}}
\newcommand \Hilb {{\rm Hilb}}
\newcommand \Sch {{\mathscr S\! }{\it ch}}
\newcommand \cHilb {{\mathscr H\! }{\it ilb}}
\newcommand \cHom {{\mathscr H\! }{\it om}}
\newcommand \colim {{{\rm colim}\, }} 
\newcommand \End {{\rm {End}}}
\newcommand \coker {{\rm {coker}}}
\newcommand \id {{\rm {id}}}
\newcommand \van {{\rm {van}}}
\newcommand \spc {{\rm {sp}}}
\newcommand \Ob {{\rm Ob}}
\newcommand \Aut {{\rm Aut}}
\newcommand \cor {{\rm {cor}}}
\newcommand \Cor {{\it {Corr}}}
\newcommand \res {{\rm {res}}}
\newcommand \red {{\rm{red}}}
\newcommand \Gal {{\rm {Gal}}}
\newcommand \PGL {{\rm {PGL}}}
\newcommand \Bl {{\rm {Bl}}}
\newcommand \Sing {{\rm {Sing}}}
\newcommand \spn {{\rm {span}}}
\newcommand \Nm {{\rm {Nm}}}
\newcommand \inv {{\rm {inv}}}
\newcommand \codim {{\rm {codim}}}
\newcommand \Div{{\rm{Div}}}
\newcommand \sg {{\Sigma }}
\newcommand \DM {{\sf DM}}
\newcommand \Gm {{{\mathbb G}_{\rm m}}}
\newcommand \tame {\rm {tame }}
\newcommand \znak {{\natural }}
\newcommand \lra {\longrightarrow}
\newcommand \hra {\hookrightarrow}
\newcommand \rra {\rightrightarrows}
\newcommand \ord {{\rm {ord}}}
\newcommand \Rat {{\mathscr Rat}}
\newcommand \rd {{\rm {red}}}
\newcommand \bSpec {{\bf {Spec}}}
\newcommand \Proj {{\rm {Proj}}}
\newcommand \pdiv {{\rm {div}}}
\newcommand \CH {{\it {CH}}}
\newcommand \wt {\widetilde }
\newcommand \ac {\acute }
\newcommand \ch {\check }
\newcommand \ol {\overline }
\newcommand \Th {\Theta}
\newcommand \cAb {{\mathscr A\! }{\it b}}

\newenvironment{pf}{\par\noindent{\em Proof}.}{\hfill\framebox(6,6)
\par\medskip}

\newtheorem{theorem}[subsection]{Theorem}
\newtheorem{conjecture}[subsection]{Conjecture}
\newtheorem{proposition}[subsection]{Proposition}
\newtheorem{lemma}[subsection]{Lemma}
\newtheorem{remark}[subsection]{Remark}
\newtheorem{remarks}[subsection]{Remarks}
\newtheorem{definition}[subsection]{Definition}
\newtheorem{corollary}[subsection]{Corollary}
\newtheorem{example}[subsection]{Example}
\newtheorem{examples}[subsection]{examples}

\title{The variation of the Gysin kernel in a family}
\author{Kalyan Banerjee}

\email{kalyan@math.tifr.res.in, kalyan.b@srmap.edu.in}
\begin{abstract}
Consider a smooth projective surface $S$. Consider a fibration $S\to C$ where $C$ is a quasi-projective curve such the fibers are smooth projective curves. The aim of this text is to show that the kernels of the push-forward homomorphism $\{j_{t*}\}_{t\in C}$ from the Jacobian $J(C_t)$ to $A_0(S)$ forms a family in the sense that it is a countable union of translates of an abelian scheme over $C$ sitting inside the Jacobian scheme $\bcJ\to C$, such that the fiber of this countable union at $t$ is the kernel of $j_{t*}$.
\end{abstract}

\maketitle

\section{Introduction}
In the paper \cite{BG} the author studies the following problem. Let $S$ be a complex projective smooth surface embedded in $\PR^N$. Let $C_t$ be a smooth hyperplane section of $S$ and $j_t$ denote the closed embedding of $C_t$ into $S$. Then $j_t$ induces a push-forward homomorphism $j_{t*}$ from $A_0(C_t)$ to $A_0(S)$, where $A_0$ is the group of algebraically trivial zero cycles modulo rational equivalence on the ambient variety. The question is what is the kernel of this homomorphism? The answer as present in the paper \cite{BG} tells us that if the geometric genus of the surface is greater than zero and the irregularity of the surface is zero then the kernel for a very general $t$ is countable. 

The main method used in proving this is the technique due to Roitman as present in \cite{R}, saying that the natural map from the symmetric powers of a surface to the Chow group of zero cycles on the surface has its fibers equal to a countable union of Zariski closed subsets. This leads to the fact that the kernel of $j_{t*}$ is a countable union of translates of an abelian variety in $J(C_t)$, where $J(C_t)$ is the Jacobian variety of $C_t$. Then since we are working with a complete linear system of curves on $S$, the presence of monodromy assures the simplicity of $J(C_t)$. This leads to the fact that $j_{t*}$ is either zero or its kernel is countable. Finally the fact that the geometric genus of the surface is greater than zero eliminates the first possibility.

The question in this paper is what happens to the kernels of $j_{t*}$ when $t$ varies. That is can we prove that either the kernel is countable for any $t$ such that $C_t$ is smooth or $j_{t*}$ is zero for any $t$. That is the kernel varies in a family in a nice algebro-geometric way. 

In this paper we answer this question: that is when we consider a very ample line bundle on a complex smooth projective surface, then the kernel of $j_{t*}$ varies in a family in the following sense: There exists a countable collection of Zariski closed subsets in the relative Jacobian scheme (the family of Jacobians of the smooth projective curves $C_t$'s where $t$ varies), such that a unique one of them is an abelian scheme say $\bcA$. Moreover the fibers of the scheme $\bcA$ at $t$ is the abelian variety whose translates gives rise to $\ker(j_{t*})$. This is done by generalizing the approach present in \cite{R}. Then due to \cite{Vo} the presence of monodromy ensures that $\bcA$ is either trivial or all of the Jacobian scheme, giving affirmative answer to our question. In the case of triviality a nice argument due to Voisin present in \cite{Vo}[theorem 7.22] tell us that the kernel is actually torsion. This applies to the case of K3 surfaces and we derive the main result of the paper:

\smallskip

\textit{Consider a $t$ in $\PR^N$ such that the corresponding fiber $C_t$ in $S$ is smooth. Suppose that the abelian variety $A_t$ is trivial and hence the kernel of $j_{t*}$ is countable. Then the kernel of $j_{t*}$ is  torsion.}

\smallskip

{\small \textbf{Acknowledgements:} The author would like to thank the hospitality of Tata Institute Mumbai, for hosting this project. }

\section{Family of gysin kernels}

Let $S$ be a smooth projective surface fibered into smooth curves over a quasi-projective curve $C$. Let for $t\in C$, $C_t$ be the fiber over $t$ and $j_t$ be the closed embedding of $C_t$ in $S$. Then we can consider the push-forward homomorphism $j_{t*}$ from $A_0(C_t)$ to $A_0(S)$, where $A_0$ denote the group of algebraically trivial zero cycles modulo rational equivalence on the ambient variety. By the Abel-Jacobi theorem $A_0(C_t)$ is isomorphic to the Jacobian $J(C_t)$. Since the curves $C_t$ is smooth for every $t$, we can consider the family of Jacobians over $C$, call it $\bcJ$. Then consider the following subset of $\bcJ$:

$$\bcR:=\{(t,z)|z\in J(C_t), j_{t*}(z)=0\}$$
Then it is immediate that $\bcR_t$ is nothing but kernel of $j_{t*}$. So $\bcR$ is the family of the kernels of the push-forward homomorphisms $j_{t*}$, as $t$ varies over $C$.

We claim that:

\begin{theorem}
$\bcR$ is a countable union of Zariski closed subschemes of $\bcJ$.
\end{theorem}

\begin{proof}
So let us consider the relation that $z_b=z_b^+-z_b^-$ supported on $\Sym^g C_t$ such that its push-forward under $j_{t*}$ is rationally equivalent to zero on $S$. That means that there exists $f:\PR^1\to \Sym^{d,d}(S)$, such that
$$f(0)=j_{t*}(z_b^{+})+\gamma,f(\infty)=j_{t*}(z_b^{-})+\gamma\;.$$
In other words we have the following map $\ev:Hom^v(\PR^1_k,\Sym^d S)$, given by $f\mapsto (f(0),f(\infty))$. That gives us the morphism $f\mapsto (f(0),f(\infty))$, from $\ev:Hom^v(\PR^1_k,\Sym^d S)$ to $ \Sym^{d,d}(S)$. Consider the closed subscheme $\bcV_{d,d}$ of $ \Sym^{d,d}(S/C)$ (this is the product of relative symmetric powers of $S$ over $C$) given by $(t,z_1,z_2)$, such that $(z_1,z_2)$ belongs to $\Sym^{d,d}(C_t)$. Then consider the map from $\bcV_{d,d}\times \Sym^{u,u}S$ to $\Sym^{d+u,u,d+u,u}S$ given by
$$(A,B,C,D)\mapsto (j_{t*}(A)+C,C,j_{t*}(B)+D,D)\;.$$
Then we can write the fiber product $\bcV$ of $Hom^v(\PR^1_k,\Sym^d S)$ and $\bcV_{d,d}\times \Sym^{u,u} S$ over $\Sym^{d+u,u,d+u,u}S$. If we consider the projection from $\bcV$ to $\Sym^{d,d}(S/C)$, then we get that $A$ and $B$ are supported  on $C_t$, and their push-forward are rationally equivalent on $S$. Conversely if $A,B$ are supported  on $C_t$ and their push-forward rationally equivalent on $S$, then we have $f:\PR^1\to \Sym^{d+u,u,d+u,u}(S)$ of some degree $v$ such that
$$f(0)=(A+C,C)\; f(\infty)=(B+D,D)\;,$$
where $C,D$ are supported on $S$. This analysis says that the image of the projection from $\bcV$ to $ \Sym^{d,d}(S/C)$ , is a quasi-projective subscheme $W_{d}^{u,v}$ consisting of tuples $(t,A,B)$ such that $A,B$ are supported on $C_t$ and there exists $f:\PR^1_k\to \Sym^{d+u,u}S$ such that $f(0)=(j_{t*}(A)+C,C)$ and $f(\infty)=(j_{t*}(B)+D,D)$, where $f$ is of degree $v$ and $C,D$ are supported on $S$ and they are co-dimension $p$ and degree $u$ cycles. So it means that $W_d=\cup_{u,v}W_d^{u,v}$. Now we prove that the Zariski closure of $W_d^{u,v}$ is in $W_d$ for each $u,v$.

For that we prove the following,

$$W_d^{u,v}=pr_{1,2}(\wt{s}^{-1}(W^{0,v}_{d+u}\times \wt{W^{0,v}_u}))$$
where
$$\wt{s}(t,A,B,C,D)=(t,j_{t*}(A)+C,j_{t*}(B)+D,C,D)$$
from $ V_{d,d}\times \Sym^{u,u}(S)$ to $\Sym^{d+u,d+u,u,u}(S)\;.$
Here $\wt{W^{u,v}_d}$ is the collection of pairs $(C,D)$ in $\Sym^{d,d}S$, such that there exists a map $f$ from $\PR^1_k$ to $\Sym^{d+u,u}S$ and $(C',D')$ in $\Sym^{u,u}S$, such that
$$f(0)=(C+C',C'),\quad f(\infty)=(D+D',D')\;.$$
Let $(t,A,B,C,D)$ be such that its image under $\wt{s}$ is in $W^{0,v}_{d+u}\times \wt{W^{0,v}_u}$. It means that there exists an element $(t,g)$ in $C\times\Hom^v(\PR^1_k,\Sym^{d+u} S)$ and another $h$ in $\Hom^v(\PR^1_k,\Sym^{u}(X))$ such that $g(0)=j_{t*}(A)+C,g(\infty)=j_{t*}(B)+D$ and $h(0)=C,h(\infty)=D$. Let us consider $f=g\times h$, then $f$ belong to $\Hom^v(\PR^1_k,\Sym^{d+u,u}(X))$, such that
$$f(0)=(A+C,C),(f(\infty))=(B+D,D)\;.$$
It means that $(t,A,B)$ belong to $W^d_{u,v}$.

On the other hand suppose that $(t,A,B)$ belongs to $W_d^{u,v}$. Then there exists $f$ in $\Hom^v(\PR^1_k,\Sym^{d+u,u}(S))$ such that
$$f(0)=(j_{t*}(A)+C,C),f(\infty)=(j_{t*}(B)+D,D)\;,$$
here $(A,B)$ belongs to $\Sym^{d+u,d+u}C_t$.
Compose $f$ with the projections to $\Sym^{d+u}(S)$ and to $\Sym^u(S)$, then we have $g$ in $\Hom^v(\PR^1_k,\Sym^{d+u}(S))$ and $h\in\Hom^v(\PR^1_k,\Sym^{u}(S))$, such that
$$g(0)=A+C,g(\infty)=B+D$$
and
$$h(0)=C,h(\infty)=D\;,$$
and we have that the image of $A,B$ are contained in the symmetric power of $C_t$. So we have
$$W_d^{u,v}\subset pr_{1,2}(\wt{s}^{-1}(W^{0,v}_{d+u}\times \wt{W^{0,v}_u})\;. $$
Therefore we have that
$$W_d=pr_{1,2}(\wt{s}^{-1}(W_{d+u}\times W_u))\;.$$

Then we prove that the closure of $W_d^{0,v} $ is contained in $W_d$. Let $(t,A,B)$ be a closed point in the closure of ${W_d^{0,v}}$. Let $W$ be an irreducible component of ${W_d^{0,v}}$ whose closure contains $(t,A,B)$. Let $U$ be an affine neighborhood of $(t,A,B)$ such that $U\cap W$ is non-empty. Then there is an irreducible curve $C$ in $U$ passing through $(t,A,B)$. Let $\bar{C}$ be the Zariski closure of $C$ in $\bar{W}$. Consider the map
$$e:  \Hom^v(\PR^1_k,\Sym^{d}(S))\to \Sym^{d,d}(S)$$
given by
$$f\mapsto (f(0),f(\infty))$$
and
$$(t,A,B)\in \Sym^{d,d}(S/C)\mapsto (j_{t*}(A),j_{t*}(B))\;. $$
Consider their fibered product, its image under the projection to second, third and fourth co-ordinate is $W_d^{0,v}$. Let us choose a curve $T$ in $U$ the fibered product, such that the closure of $e(T)$ is $\bar C$.  Consider the normalization $\wt{T}$ of the Zariski closure of $T$. Let $\wt{T_0}$ be the pre-image of $T$ in the normalization. Now the regular morphism $\wt{T_0}\to T\to \bar C$ extends to a regular morphism from $\wt{T}$ to $\bar C$. Now let $(f,t,A,B)$ be a pre-image of $(t,A,B)$. Then we have $f(0)=j_{t*}(A);, f(\infty)=j_{t*}(B)$. Therefore the push-forward of $A,B$ are rationally equivalent.

On the other hand by Roitman's argument \cite{R}, we have that $\wt{W^{0,v}_u}$ is closed, hence $W_d^{u,v}$ is closed. That finishes the proof.

\end{proof}

\begin{theorem}
For every $t$, the kernel of $j_{t*}$ is a countable union of translates of an abelian variety.
\end{theorem}

\begin{proof}
Follows from \cite{BG}[proposition 6].
\end{proof}

So we have that $\bcR_t$ is a countable union of translates of an abelian variety $A_t$ for all $t$.

\begin{theorem}
Consider the family $\bcO$ of zero cycles supported on $\bcJ$, given by
$$\bcO=\{(t,z)| z\in J(C_t), z=0\;.\}$$
Then through this family admits a morphism to  a unique component of $\bcR$, which is an abelian scheme.
\end{theorem}

\begin{proof}
Suppose that there passes $m$-many components of $\bcR$ which admits a morphism from the family of elements $\bcO$ mentioned above. Call them $\bcR_1,\cdots,\bcR_m$. So given any $t\in C$, we have ${\bcR_i}_t$ passes through the trivial element $\bcO_t$ in $J(C_t)$ for all $i$. So in particular for the generic point $\eta$ of $C$. Therefore we have that $\sum_i \bcR_{i\eta}$ is contained in the kernel of $j_{\eta*}$ and it contains $\bcR_{i\eta}$ for all $i$. This is impossible by the irredundancy of the decomposition of $\bcR$. Therefore there exists a unique $\bcR_0$, whose generic fiber passes through the zero in $J(C_\eta)$. Hence $\bcR_0$ admits a morphism  from the family of elements $\bcO$ given above and it is the unique one with this property. Following the argument in \cite{BG}, we can prove that $\bcR_{0\eta}$ is an abelian variety. Now we prove that $\bcR_0$ is an abelian scheme. For that we have to prove that $\bcR_0$ is a group scheme over $C$. We have $\bcJ$ an abelian scheme. So there exist a group operation $\mu:\bcJ\times_C \bcJ\to \bcJ$. Restrict this operation to $\bcR_0\times_C \bcR_0$. Then the image of this restriction map is a Zariski closed subscheme of $\bcJ$, admitting a map from the family $\bcO$ and also it is contained in $\bcR$. Because $\bcR$ is closed under the operation $\mu$. Since the ground field $k$ is uncountable, there exists a unique component $\bcR_i$, of $\bcR$, which contains the image. Also $\bcR_i$ should admit a morphism from $\bcO$. Hence by uniqueness of $\bcR_0$, we have that $\bcR_i=\bcR_0$. Therefore the image of $\mu$ from $\bcR_0\times_C \bcR_0$, is in $\bcR_0$. Consider the inverse map $i:\bcJ\to \bcJ$. We have to prove that $\bcR_0$ is closed under $i$. Consider the image $i(\bcR_0)$, this admits a morphism from $\bcO$, hence it must land inside $\bcR_0$, by the uniqueness. So we have that $\bcR_0$ is an abelian scheme. Also $\bcR_{0,t}$, therefore is an abelian variety passing through the trivial element of $J(C_t)$, hence by uniqueness it is the abelian variety $A_t$, whose translates cover $\ker(j_{t*})$.
\end{proof}

\section{Application of the above result in presence of monodromy}

Consider the following situation. We have a smooth projective surface $S$ over the field of complex numbers which  is embedded in some $\PR^N$. Then consider a Lefschetz pencil on $S$. That is we have a fibration $S\to \PR^1$, with the singular fibers having atmost one ordinary double point. Then consider the Zariski open set $U=\PR^1\setminus 0_1,\cdots,0_m$, where $0_i$'s corresponds to the singular fibers. We have the family
$$\wt{S}_U:=\{(s,t):s\in C_t, t\in U\}$$
where $C_t$ is the smooth fiber over $t\in U$. Then we have the monodromy action of $\pi_1(U,0)$, $0$ fixed in $U$, on the vanishing cohomology
$$\ker(j_{t*}:H^1(C_t,\QQ)\to H^3(S,\QQ))\;.$$
Now consider the kernel of the push-forward $j_{t*}$ in a family. Then by the above we have an abelian scheme $\bcA\subset \bcJ$ over $U$. Infact $\bcA\subset \bcA_{van}$, where $\bcA_{van}$ is the abelian scheme corresponding to the variation of Hodge structures given by
$$\bcH_{van}=\ker(j_{t*}:H^1(C_t,\QQ)\to H^3(S,\QQ))\;.$$
By the equivalence of variations of Hodge structures of weight one and abelian schemes, $\bcA$ corresponds to a sub-variation of Hodge structures $\bcH$ in $\bcH_{van}$, which gives rise to a local system, contradicting the irreducibility of the local system $\bcH_{van}$. Therefore $\bcA$ is either trivial or all of $\bcA_{van}$. In particular when the irregularity of the surface is zero, then $\bcA$ is either trivial or all of $\bcJ$. So we prove that for any $t$ in ${\PR^N}^*$ such that $C_t$ is smooth we have $A_t=0$ or $J(C_t)$ and moreover it happens in a family: meaning that:

\begin{proposition}
For all smooth fibers $C_t$, $A_t$ is zero or for all smooth fibers $C_t$, $A_t$ is $J(C_t)$.
\end{proposition}

It is interesting when the abelian variety $A_t$ is trivial for all $t$, such that $C_t$ is smooth. This is because in this case the trivial abelian variety comes from the trivial abelian subscheme of $\bcJ$. So it corresponds to a global section of the local system $H^{1}(C_t,\CC)/H^1(C_t,\ZZ)$. Since this local system has the action of $\pi_1(U,t)$, where $U$ is the set in ${\PR^N}^*$ parametrising smooth hyperplane sections of $S$. Let $\beta$ be such a section. Then for $\gamma$ in $\pi_1(U,t)$, we have that $\rho(\gamma)(\beta)$ is again in $H^1(C_t,\CC)/H^1(C_t,\ZZ)$. This means that for a lift $\wt{\beta}$ of $\beta$ to $H^1(C_t,\CC)$ we have
$$\rho(\gamma)(\wt{\beta})-\wt{\beta}\in H^{1}(C_t,\ZZ)\;.$$
By the Picard-Lefschetz formula the action of $\rho$ on $\Aut(H^1(C_t,\CC))$ is given by
$$\rho(\gamma_i)(\wt{\beta})=\beta\pm\langle \wt{\beta}, \delta_i\rangle \delta_i$$
where $\gamma_i$'s are the generators of the group $\pi_1(U,t)$ and $\delta_i$ is the corresponding vanishing cycle. Then by the above we have that
$$\langle \wt{\beta}, \delta_i\rangle \delta_i$$
is in $H^{1}(C_t,\ZZ)$. Therefore $\langle \wt{\beta}, \delta_i\rangle$ is in $\ZZ$, hence in $\QQ$. Therefore $\wt{\beta}$ is in $H^1(C_t,\QQ)$, as the vanishing cycles generate $H^1(C_t,\QQ)$. Therefore $\beta$ is torsion. Consider the case when the abelian scheme $\bcA$ is trivial. Then all the elements in $\ker(j_{t*})$
corresponds to a global section of the local system $H^1(C_t,\CC)/H^1(C_t,\ZZ)$. So we have that the kernel of $j_{t*}$ is torsion for a general $t$. So we proved the following theorem:

\begin{theorem}
Consider a $t$ in $U$, such that the corresponding fiber $C_t$ in $S$ is smooth. Suppose that the abelian variety $A_t$ is trivial and hence the kernel of $j_{t*}$ is countable. Then the kernel of $j_{t*}$ is  torsion.
\end{theorem}

\begin{corollary}
Consider a hyperplane section of a K3 surface which is smooth. Then the corresponding Gysin kernel is torsion.
\end{corollary}

\end{document}